\documentclass{article}
\usepackage{authblk}
\usepackage{abstract}
\usepackage{amsmath}
\usepackage{wasysym}
\usepackage{amsthm}
\usepackage{amssymb}
\newtheorem{definition}{Definition}
\newtheorem{Proposition}{Proposition}
\newtheorem{Lemma}{Lemma}
\newtheorem{Theorem}{Theorem}
\newtheorem{Remark}{Remark}
\newtheorem{Example}{Example}

\begin{document}
	\title{On the equivalent $p$-th von Neumann-Jordan constant associated with isosceles orthogonality in Banach spaces}
	\author[1]{Yuxin Wang}
	\author[1]{Qi Liu\thanks{Qi Liu:liuq67@aqnu.edu.cn }}
	\author[1]{Yongmo Hu}
	\author[1]{Jinyu Xia}
	\author[1]{Mengmeng Bao}
	
	\affil{School of Mathematics and physics, Anqing Normal University, Anqing 246133, P.R.China}
	\maketitle 
	\begin{abstract}
		In this paper, we define a new geometric constant based on isosceles orthogonality, denoted by $C^I_{NJ}(\alpha,p,X).$ Through research, we find that this constant is the equivalent $p$-th von Neumann–Jordan constant in the sense of isosceles orthogonality. First, we obtain some basic properties of the constant. Then, we calculate the upper and lower bounds of the constant. Through three examples, it is found that the upper bound of the constant is attainable. We also compare the relationship between this constant and other constants. Finally, we establish the connection between the constant and some geometric properties in Banach spaces, such as uniform non-squareness, uniform smoothness.
	\end{abstract}
	\textbf{keywords:} {isosceles orthogonality; $p$-th von Neumann-Jordan constant; uniformly non-square; uniformly smooth}
	
	\section{Introduction} 
	The geometric theory of Banach spaces, a vital branch within functional analysis, holds significant importance across numerous mathematical fields such as approximation theory and fixed point theory. Following Clarkson's \cite{07} introduction of the convexity modulus concept in 1936, geometric constants have become a powerful means of evaluating and contrasting the geometric properties of Banach spaces. Quantifying these geometric attributes through numerical values provides an intuitive and accessible method for comprehending the characteristics of a specific Banach space. In the realm of investigating the geometric structure of Banach spaces, numerous recent studies have centered their attention on the von Neumann–Jordan  constant. It is widely acknowledged that this  constant maintains a strong connection with several geometric structures, including uniform nonsquareness and uniform normal structure \cite{04,05,06}.

In Euclidean geometry, orthogonality is a fundamental concept embodied by perpendicularity via the inner product. However, this changes in Banach spaces: unlike Euclidean spaces with a natural inner product, Banach spaces lack this structure, eliminating the canonical orthogonality definition. To bridge this gap, scholars have introduced various generalized orthogonality concepts, aimed at extending geometric orthogonality intuition while adapting to Banach spaces' diverse structures. Below are several such definitions:

(i) Birkhoff–James Orthogonality ($x \perp_B x_2$) \cite{09}:
A vector $x_1$ is Birkhoff–James orthogonal to $x_2$ if and only if
$$\|x_1 + \alpha y\| \geq \|x_1\|$$ for all $\alpha\in\mathbb{R}.$

(ii) Isosceles Orthogonality ($x_1 \perp_I x_2$) \cite{02}:
Vectors $x_1$ and $x_2$ are isosceles orthogonal if and only if
\[
\|x_1 + x_2\| = \|x_1 - x_2\|.
\]  

(iii) Pythagorean Orthogonality ($x_1 \perp_P x_2$) \cite{02}:
A vector $x_1$ is pythagorean orthogonal to $x_2$ if and only if
\[
\|x_1 - x_2\|^2 = \|x_1\|^2 + \|x_2\|^2.
\]  

(iv) Roberts Orthogonality ($x \perp_R y$) \cite{12}: Vectors $x_1$ and $x_2$ are Roberts orthogonal if and only if  $$\|x_1+\alpha x_2\|=\|x_1-\alpha x_2\|$$ for all $\alpha\in\mathbb{R}$.

\section{Preliminaries}
Before we start, we have the following agreements: for a given $\operatorname{dim}X\geq2$ Banach space \( X \), the notations \( B_X \) and \( S_X \) are employed to respectively denote the unit ball of \( X \), the unit sphere of \( X \).

 The James constant, as defined by Gao and Lau in reference \cite{08}, can be expressed through two equivalent formulations. The first definition is given by:
$$J(X) = \sup\left\{\min\left\{\|x_1+ x_2\|, \|x_1 - x_2\|\right\}: x_1, x_2 \in S_{X}\right\}.$$
An alternative definition based on isosceles orthogonality ($\perp_I$) is stated as:
$$J(X) = \sup\left\{\|x_1 + x_2\|:  x_1, x_2 \in S_{X}, x_1 \perp_I x_2\right\}.$$
The Schäffer constant is defined by the following expression:  
$$S(X)=\inf\left\{\|x_1+x_2\|:x_1 \perp_I x_2, x_1, x_2 \in S_{X}\right\}.$$  
In reference \cite{05}, it was demonstrated that $S(X)=\frac{2}{J(X)}$.  

The modulus of smoothness \cite{10} of a space \( X \), denoted by the function \( \rho(t) \), is defined as follows:
\[
\rho(t) = \sup \left\{ \frac{\|x_1 + tx_2\| + \|x_1 - tx_2\|}{2} - 1: x_1,x_2  \in S_X \right\}.
\]
For a Banach space \((X, \|\cdot\|)\), if \(\lim\limits_{t \to 0^+} \frac{\rho_X(t)}{t} = 0\) holds, then the space \((X, \|\cdot\|)\) is said to be uniformly smooth.  Furthermore, a space \( X \) is called uniformly non-square \(\cite{17}\) if \(\exists \, \delta > 0\), such that \(\forall x_1, x_2 \in S_X\),  
$$\min\bigg\{\|x_1+ x_2\|,\|x_1 - x_2\| \bigg\}\leq 2(1 - \delta).$$

In reference \cite{20}, Yang proposed the function $\gamma_X(t): [0,1] \rightarrow [1,4]$. This function is defined as 
$$\gamma_X(t) = \sup \left\{ \frac{\|x_1+tx_2\|^2 + \|x_1-tx_2\|^2}{2} : x_1,x_2 \in S_X\right\}.$$

The von Neumann-Jordan constant \( C_{NJ}(X) \), introduced by Clarkson in \cite{06}, is defined as:  
	$$C_{NJ}(X)=\sup\left\{\frac{\|x_1+x_2\|^2+\|x_1-x_2\|^2}{2\left(\|x_1\|^2+\|x_2\|^2\right)}:x_1,x_2\in X,(x_1,x_2)\neq(0,0)\right\}.$$
Clearly, there is an equation relationship between  \( C_{NJ}(X) \) and $\gamma_X(t)$, that is$$
C_{\mathrm{NJ}}(X)=\sup \left\{\frac{\gamma_X(t)}{1+t^2}: 0 \leqslant t \leqslant 1\right\} .
$$
Futhermore, the modified von Neumann-Jordan constant  is defined as
$$C'_{NJ}(X)=\sup\left\{\frac{\|x_1+x_2\|^{2}+\|x_1-x_2\|^{2}}{4}: x_1, x_2 \in S_{X}\right\}.$$

In a recent work, Cui, Huang, Hudzik, and Kaczmarek \cite{16} proposed a new geometric constant named the generalized von Neumann-Jordan constant, denoted by \( C_{NJ}^{(p)}(X) \). For a normed space \( X \) and \( 1 \leq p < \infty \), this constant is defined as:  
\[
C_{NJ}^{(p)}(X)= \sup \left\{ \frac{\|x_1 + x_2\|^p + \|x_1 - x_2\|^p}{2^{p-1} \left( \|x_1\|^p + \|x_2\|^p \right)} : x_1, x_2 \in X, \ (x_1, x_2) \neq (0, 0) \right\}.
\]  
This parametrization highlights the role of the scaling parameter \( t \) (ranging over \([0, 1]\)) between unit vectors \( x \) and \( y \), enabling a more detailed analysis of the constant's behavior across different normed space structures. The parameter \( p \) generalizes the concept, with \( p = 2 \) corresponding to the classical von Neumann-Jordan constant.

Later, in \cite{21}, the authors introduced the following constant
on the basis of the $p$-th von Neumann-Jordan constant:
$$J_{X,p}(t) = \sup \left\{ \left( \frac{\| x_1 + t x_2 \|^{p} + \| x_1 - t x_2 \|^{p}}{2} \right)^{1/p}:x_1,x_2\in S_X \right\},$$ where  $t > 0$ and $1 \leq p < \infty$.

In \cite{03}, the authors studied roundness properties of Banach spaces using constant $\nu_X(p)$. This also indicates that $p$-th von Neumann type constants play an important role in dealing with other problems in Banach space theory. 
	$$
	\nu_X(p)=\sup \left\{\frac{\|x_1+x_2\|^p+\|x_1-x_2\|^p}{\|x_1\|^p+\|x_2\|^p}: x_1, x_2 \in X,(x_1,x_2)\neq(0,0)\right\}.
	$$

	In \cite{22}, the authors enumerate many properties of isosceles orthogonality. Through Theorem \textbf{4.33} in \cite{22}, it can be shown that the study of adding isosceles orthogonality conditions to geometric constants is meaningful. Therefore,	the geometric constants involving isosceles orthogonality conditions have undergone thorough academic exploration, with detailed references available in the relevant literature \cite{13,14,15}.  

Importantly, in the literature \cite{18}, the authors defined the constant 
	$$
	\Omega^{\prime}(X)=\sup \left\{\frac{\|x_1+2 x_2\|^2+\|2 x_1+x_2\|^2}{5\|x_1+x_2\|^2}: x_1 \perp_I x_2\right\} 
	$$
	and derived the equation 
	$\Omega^{\prime}(X)=\frac{9}{10} \gamma_X\left(\frac{1}{3}\right)$. Based on  constant $\Omega^{\prime}(X)$,  the author proposed a novel constant and derived the equivalent form of the von Neumann constant by means of this newly defined constant \cite{19}. 

Through the study of the above articles, we propose an idea: can we define an isosceles orthogonality-related constant equivalent to the \( p \)-th von Neumann-Jordan constant? The answer is affirmative, as can be seen in this article.
	\section{The constant $C_{N J}^I(\alpha, p, X)$}
	Before starting, 
	we introduce an important constant for use in the subsequent sections. It's known that there is an equivalent form of the $C^p_{NJ}(X)$ constant, that is$$C_{NJ}^{(p)}(X)=\sup\bigg\{\frac{\|x_1+tx_2\|^{p}+\| x_1-tx_2\|^{p}}{2^{p-1}(1+t^{p})}:x_1,x_2\in S_{X},0\leq t\leq1\bigg\},$$ where $1\leq p<\infty$.	Write $$\gamma^p_X(t)=\sup\left\{\frac{\|x_1+tx_2\|^p+\|x_1-tx_2\|^p}{2^{p-1}}:x_1,x_2\in S_X,0\leq t\leq1\right\},$$where $1\leq p<\infty$.
	
	Although the $\gamma^p_X(t)$ constant is contained in  the $C_{NJ}^{(p)}(X)$ constant, scholars have not discussed this constant separately before. Therefore, in this article, we will discuss some properties of the $\gamma^p_X(t)$ constant.
	\begin{Remark}
		(i) If $t=0$, then $\gamma^p_X(0)=\frac{1}{2^{p-2}}$ is valid for any Banach space.
		
		(ii) If $p=2$, then $$\gamma^2_X(t)=\gamma_X(t)=\sup\left\{\frac{\|x_1+tx_2\|^2+\|x_1-tx_2\|^2}{2}:x_1,x_2\in S_X,0\leq t\leq1\right\},$$
	\end{Remark}
	\begin{Lemma}\label{l2}
		Consider the function $\psi(r)=\|rx_1+tx_2\|^p+\|rx_1-tx_2\|^p$, and the parameter $t$ in this function is fixed, then $\psi(r)$ is an even and convex function.
	\end{Lemma}
	\begin{proof}
			It's obviously that $\psi(r)$ is an even function. Then we prove that $\psi(r)$ is a convex function, let $r_1,r_2\in[0,1]$ and $\lambda\in(0,1)$, then by the convexity of $\|\cdot\|^p$, we have
		$$\begin{aligned}
			&	\psi(\lambda r_1+(1-\lambda)r_2)\\&=\|(\lambda r_1+(1-\lambda)r_2)x_1+tx_2\|^p+\|(\lambda r_1+(1-\lambda)r_2)x_1-tx_2\|^p\\&=\|\lambda(r_1x_1+t x_2)+(1-\lambda)(r_2x_1+tx_2)\|^p+\|\lambda(r_1x_1-t x_2)+(1-\lambda)(r_2x_1-tx_2)\|^p\\&\leq\lambda(\|r_1x_1+tx_2\|^p+\|r_1x_1-tx_2\|^p)+(1-\lambda)(\|r_2x_1+tx_2\|^p+\|r_2x_1-tx_2\|^p)\\&=\lambda\psi(r_1)+(1-\lambda)\psi(r_2),
		\end{aligned}$$it follows that $\psi(r)$ is a convex function, as desired.

	\end{proof}
	\begin{Proposition}\label{p1}
		Let $X$ be a Banach space. Then, $$\begin{aligned}\gamma^p_{X}(t)&=\sup\left\{\frac{\|x_1+tx_2\|^{p}+\|x_1-tx_2\|^{p}}{2^{p-1}}{:}x_1\in S_{X},x_2\in B_{X},0\leq t\leq 1\right\}\\&=\sup\left\{\frac{\|x_1+tx_2\|^{p}+\|x_1-tx_2\|^{p}}{2^{p-1}}{:}x_1,x_2\in B_{X},0\leq t\leq 1\right\}.\end{aligned}$$
	\end{Proposition}
	\begin{proof}
		Clearly, $\phi(t)=\|x_1+tx_2\|^{p}+\|x_1-tx_2\|^{p}$ is a convex function of $t$ on $[0,1]$. Then let $0<t_1\leq t_2\leq1$ and $x_1,x_2\in S_X$, we have$$\begin{aligned}\|x_1+t_{1}x_2\|^{p}+\|x_1-t_{1}x_2\|^{p}&=\phi(t_{1})=\phi\left(\frac{t_{2}+t_{1}}{2t_{2}}t_{2}+\frac{t_{2}-t_{1}}{2t_{2}}(-t_{2})\right)\\&\leq \phi(t_{2})=\|x_1+t_{2}x_2\|^{p}+\|x_1-t_{2}x_2\|^{p},\end{aligned}$$which means that $\gamma^p_{X}(t_1)\leq\gamma^p_{X}(t_2)$. Hence, we obtain that $$\frac{1}{2^{p-1}}\sup_{x_1\in S_X}\sup_{x_2\in B_X}\left\{\|x_1+tx_2\|^p+\|x_1-tx_2\|^p\right\}=\gamma^p_X\left(t\|x_2\|\right)\leq\gamma^p_X(t).$$
		Then since $$\gamma^p_{X}(t)\leq\sup\left\{\frac{\|x_1+tx_2\|^{p}+\|x_1-tx_2\|^{p}}{2^{p-1}}{:}x_1\in S_{X},x_2\in B_{X},0\leq t\leq 1\right\}$$ is obvious, we obtain the first identity.
		
		Now we prove the second identity. By Lemma \ref{l2}, $\psi(r)$ is an  even and convex function. Thus, we have $\psi(r) \geq g(1)$ for all $r\geq 1$. Then let $x_1$, $x_2 \in B_X$, we have
		$$\left\|\frac{x_1}{\|x_1\|} + tx_2\right\|^p + \left\|\frac{x_1}{\|x_1\|} - tx_2\right\|^p \geq \|x_1 + t x_2\|^p + \|x_1 - tx_2\|^p.$$ Thus, we obtain that$$\begin{aligned}
			&\sup\left\{\frac{\|x_1+tx_2\|^{p}+\|x_1-tx_2\|^{p}}{2^{p-1}}{:}x_1\in S_{X},x_2\in B_{X},0\leq t\leq 1\right\}\\\geq&\sup\left\{\frac{\|x_1+tx_2\|^{p}+\|x_1-tx_2\|^{p}}{2^{p-1}}{:}x_1\in B_{X},x_2\in B_{X},0\leq t\leq 1\right\}.
		\end{aligned}$$
		On the other hand, it's obviously that $$\begin{aligned}
			&\sup\left\{\frac{\|x_1+tx_2\|^{p}+\|x_1-tx_2\|^{p}}{2^{p-1}}{:}x_1\in S_{X},x_2\in B_{X},0\leq t\leq 1\right\}\\\leq&\sup\left\{\frac{\|x_1+tx_2\|^{p}+\|x_1-tx_2\|^{p}}{2^{p-1}}{:}x_1\in B_{X},x_2\in B_{X},0\leq t\leq 1\right\}.
		\end{aligned}$$
		Therefore, we have completed the entire proof.
	\end{proof}

Through the observation and study of the \( p \)-th von Neumann-Jordan constant, we found that a new constant based on isosceles orthogonality can be defined, and this new constant is equivalent to the \( p \)-th von Neumann-Jordan constant, the main definition of this article is given below:
		\begin{definition}
		Let $X$ be a Banach space. For $1\leq p<\infty$, $$\begin{aligned}
			&C^{I}_{NJ}(\alpha,p,X)\\=&\sup\left\{\frac{\|\alpha x_1+(1-\alpha )x_2\|^{p}+\|(1-\alpha)x_1+\alpha x_2\|^{p}}{\|x_1+x_2\|^{p}}:x_1,x_2\in X,(x_1,x_2)\neq(0,0),x_1\perp_{I}x_2\right\},
		\end{aligned}$$ where $0\leq \alpha\leq\frac{1}{2}.$
	\end{definition}
	\begin{Remark}
		 If $\alpha=\frac12$, then $C^p_{I}(\frac12,X)=\frac{1}{2^{p-1}}$ is valid for any Banach space.
	\end{Remark}
	Then the following proposition provides  some basic properties of this new constant.
	\begin{Proposition}
			Let $X$ be a Banach space. Then,
		
			(i) $C^I_{NJ}(\alpha,p,X)$ is a convex function of $\alpha$ on $[0,\frac12]$.
			
			(ii) $C^I_{NJ}(\alpha,p,X)$ is a non-decreasing function of $\alpha$ on $[0,\frac12]$.
			
			(iii) $C^I_{NJ}(\alpha,p,X)$ is continuous of $\alpha$ on $[0, \frac12]$.
		\end{Proposition}
		\begin{proof}
			(i)To prove that $C^I_{NJ}(\alpha,p,X)$ is a convex function, let $\alpha\in[0,\frac12]$ and $1\leq p<\infty$, we define $$F(\alpha):=\|\alpha x_1+(1-\alpha)x_2\|^p+\|(1-\alpha)x_1+\alpha x_2\|^p.$$
			 Then by fixing \(x_1, x_2\) on \(X\) such that \(x_1 \perp_I x_2\), we only need to show that $F(\alpha)$ is a convex function of $\alpha$ on $[0,\frac12]$.
			 Let $\alpha_1,\alpha_2\in[0,\frac12]$ and $t\in(0,1)$, by the convexity of $\|\cdot\|^p$, we have
			$$\begin{aligned}
				&F(t \alpha_1+(1-t)\alpha_2)\\=&\big\|(t \alpha_1+(1-t)\alpha_2)x_1+\big(1-(t \alpha_1+(1-t)\alpha_2)\big)x_2\big\|^p\\&+\big\|\big(1-(t \alpha_1+(1-t)\alpha_2)\big)x_1+(t \alpha_1+(1-t)\alpha_2)x_2\big\|^p\\=&\big\|t(\alpha_1x_1+(1-\alpha_1)x_2)+(1-t)(\alpha_2x_1+(1-\alpha_2)x_2)\big\|^p\\&+\big\|t\big((1-\alpha_{1})x_1+\alpha_{1}x_2\big)+(1-t)\big((1-\alpha_{2})x_1+\alpha_{2}x_2\big)\big\|^p\\\leq&t\|\alpha_1x_1+(1-\alpha_1)x_2\|^p+(1-t)\|\alpha_2x_1+(1-\alpha_2)x_2\|^p\\&+t\|(1-\alpha_{1})x_1+\alpha_{1}x_2\|^p+(1-t)\|(1-\alpha_{2})x_1+\alpha_{2}x_2\|^p\\=&t(\|\alpha_1x_1+(1-\alpha_1)x_2\|^p+\|(1-\alpha_{1})x_1+\alpha_{1}x_2\|^p)\\&+(1-t)(\|\alpha_2x_1+(1-\alpha_2)x_2\|^p+\|(1-\alpha_{2})x_1+\alpha_{2}x_2\|^p)\\=&tF(\alpha_1)+(1-t)F(\alpha_2),
			\end{aligned}$$it follows that $F(\alpha)$ is a convex function, as desired.
		
		(ii)Similarly, to prove that  $C^I_{NJ}(\alpha,p,X)$ is a non-decreasing function. We only need to prove that the $F(\alpha)$ in (i) is a non-decreasing function. Since $F(\alpha)$ is a convex function of $\alpha$ on $[0,\frac12]$. Then let $0\leq\alpha_1\leq \alpha_2\leq\frac12$ and $x_1,x_2\in X$ such that $x_1\perp_I x_2$, we have$$\begin{aligned}&\|\alpha_1 x_1+(1-\alpha_1)x_2\|^p+\|(1-\alpha_1)x_1+\alpha_1 x_2\|^p=F(\alpha_{1})\\&=F\left(\frac{\alpha_{2}+\alpha_{1}}{2\alpha_{2}}\alpha_{2}+\frac{\alpha_{2}-\alpha_{1}}{2\alpha_{2}}(-\alpha_{2})\right)\\&\leq F(\alpha_{2})=\|\alpha_2 x_1+(1-\alpha_2)x_2\|^p+\|(1-\alpha_2)x_1+\alpha_2 x_2\|^p,\end{aligned}$$which means that $F(\alpha)$ is a non-decreasing function.  
		
		(iii)Obviously.
		\end{proof}
		\begin{Proposition}\label{pp}
		Let $X$ be a Banach space with $0\leq\alpha\leq\frac12$ and $1\leq p<\infty$. Then \begin{equation}\label{e1}
			(1-\alpha)^p+\alpha^p\leq C^I_{NJ}(\alpha,p,X)\leq2(1-\alpha)^p.
		\end{equation}
	\end{Proposition}
	\begin{proof}
		First, let $x_1=0, x_2\neq0$, then we have $x_1\perp_ I x_2$ and $$
		\begin{aligned}
			\frac{\left\|\alpha x_1+(1-\alpha) x_2\right\|^p+\left\|(1-\alpha) x_1+\alpha x_2\right\|^p}{\left\|x_1+x_2\right\|^p} & =\frac{\left\|(1-\alpha) x_2\right\|^p+\left\|\alpha x_2\right\|^p}{\left\|x_2\right\|^p} \\
			& =(1-\alpha)^p+\alpha^p.
		\end{aligned}
		$$
		Thus, we obtain that $ C^I_{NJ}(\alpha,p,X)\geq(1-\alpha)^p+\alpha^p$.
		
		Convesely, let $x_1,x_2\in X$ such that $x_1\perp_ I x_2$. Since$$
		\begin{aligned}
			& \frac{\left\|\alpha x_1+(1-\alpha) x_2\right\|^p+\left\|(1-\alpha) x_1+\alpha x_2\right\|^p}{\left\|x_1+x_2\right\|^p} \\
			=& \frac{\left\|\frac12\left(x_1+x_2\right)-\frac{1-2 \alpha}{2}\left(x_1-x_2\right)\right\|^p+\left\|\frac{1}{2}\left(x_1+x_2\right)+\frac{1-2 \alpha}{2}\left(x_1-x_2\right)\right\|^p}{\left\|x_1+x_2\right\|^p} \\\leq
			& \frac{2\cdot\left[\frac{1}{2}\left\|x_1+x_2\right\|+\frac{1-2 \alpha}{2}\left\|x_1-x_2\right\|\right]^p}{\left\|x_1+x_2\right\|^p} \\
			=& 2(1-\alpha)^p.
		\end{aligned}
		$$
		This imples that $ C^I_{NJ}(\alpha,p,X)\leq2(1-\alpha)^p.$
	\end{proof}
	\begin{Example}\label{e}
Let $X_1=(\mathbb R^n,\|\cdot\|_1)$. Then,$$C^I_{NJ}(\alpha,p, X_1)=2(1-\alpha)^p.$$
	\end{Example}
	\begin{proof}
		Let $x_1=(1,1,0,\cdots),x_2=(1,-1,0,\cdots)$, then we have $$\|x_1+x_2\|_1=\|x_1-x_2\|_1=2,$$which means that $x_1\perp_ I x_2$. Furthermore, we can get that $$\|\alpha x_1+(1-\alpha )x_2\|_1=\|(1-\alpha)x_1+\alpha x_2\|_1=2-2\alpha.$$ Thus, we obtain that $$\frac{\|\alpha x_1+(1-\alpha )x_2\|_1^{p}+\|(1-\alpha)x_1+\alpha x_2\|_1^{p}}{\|x_1+x_2\|_1^{p}}=2(1-\alpha)^p.$$This implies that $C^I_{NJ}(\alpha,p, X_1)\geq2(1-\alpha)^p.$ Then by \eqref{e1} in Proposition \ref{pp}, we can get the result.
	\end{proof}
		\begin{Example}\label{ee}
		Let $X_2=(\mathbb R^n,\|\cdot\|_\infty)$. Then,$$C^I_{NJ}(\alpha,p, X_2)=2(1-\alpha)^p.$$
	\end{Example}
	\begin{proof}
		Let $x_1=(1,0,0,\cdots),x_2=(0,-1,0,\cdots)$, then we have $$\|x_1+x_2\|_\infty=\|x_1-x_2\|_\infty=1,$$which means that $x_1\perp_ I x_2$. Furthermore, we can get that $$\|\alpha x_1+(1-\alpha )x_2\|_\infty=\|(1-\alpha)x_1+\alpha x_2\|_\infty=1-\alpha.$$ Thus, we obtain that $$\frac{\|\alpha x_1+(1-\alpha )x_2\|_\infty^{p}+\|(1-\alpha)x_1+\alpha x_2\|_\infty^{p}}{\|x_1+x_2\|_\infty^{p}}=2(1-\alpha)^p.$$This implies that $C^I_{NJ}(\alpha,p, X_2)\geq2(1-\alpha)^p.$ Then by \eqref{e1} in Proposition \ref{pp}, we can get the result.
	\end{proof}
	\begin{Example}
	Consider the linear space $C[a,b]$ consisting of all real-valued continuous functions, equipped with the norm $\|f\| = \sup\limits_{t\in[a,b]} |f(t)|$. Then $$C^I_{NJ}(\alpha,p,X)=2(1-\alpha)^p.$$
	\end{Example}
	\begin{proof}
		Let $f_1 = \frac{t-b}{a-b}, f_2 = 1-\frac{t-b}{a-b} \in S_{C[a,b]}$, then we have $$\|f_1+f_2\|=\|f_1-f_2\|=1,$$ which means that $f_1 \perp_I f_2.$ Furthermore, we have $$\|\alpha f_1+(1-\alpha )f_2\|=\sup_{t\in[a,b]} \left| \frac{2\alpha-1}{a-b}(t-b) + 1 - \alpha \right|=1-\alpha,$$ and $$\|(1-\alpha) f_1+\alpha f_2\|=\sup_{t\in[a,b]} \left| \frac{1-2\alpha}{a-b}(t-b) + \alpha \right|=1-\alpha,$$ 
	This implies that $C^I_{NJ}(\alpha,p,X)\geq2(1-\alpha)^p.$ Then by \eqref{e1} in Proposition \ref{pp}, we can get the result.
	\end{proof}
\begin{Theorem}\label{cr}
	Let $X$ be a Banach space with $0\leq\alpha\leq\frac12$ and $1\leq p<\infty$. Then,$$ C^I_{NJ}(\alpha,p,X)=\frac12\gamma^p_X(1-2\alpha).$$
\end{Theorem}
\begin{proof}
First, if $\alpha=\frac{1}{2}$, then $C^p_{I}(\frac12,X)=\frac12\gamma^p_X(0)$ is valid for any Banach spaces.

Now we consider the situation that $0\leq\alpha<\frac12$. Let $x_1, x_2 \in X$ such that $x_1\perp_I x_2$, put $u_1 = \frac{x_1+x_2}{2}$, $u_2 = \frac{x_1-x_2}{2}$, then  $\|u_1\| = \|u_2\|$ and
	$$\alpha x_1 + (1-\alpha)x_2 = u_1 - (1-2\alpha)u_2, (1-\alpha)x_1 + \alpha x_2 = u_1 + (1-2\alpha)u_2,$$
	Now let $x'_1=\frac{u_1}{\|u_1\|}$ and $x'_2=\frac{u_2}{\|u_1\|}$, thus, $x'_1,x'_2\in S_X$ and we get that
	$$\begin{aligned}
			\frac{\|\alpha x_1 + (1-\alpha)x_2\|^p + \|(1-\alpha)x_1 + \alpha x_2\|^p}{\|x_1+x_2\|^p} &= \frac{\|u_1- (1-2\alpha)u_2\|^p + \|u_1 + (1-2\alpha)u_2\|^p}{2^p\|u_1\|^p}\\&=\frac{\left\|x^{\prime}_1-(1-2\alpha)x^{\prime}_2\right\|^{p}+\left\|x^{\prime}_1+(1-2\alpha)x^{\prime}_2\right\|^{p}}{2^p}\\&\leq\frac12\gamma^p_X(1-2\alpha).
	\end{aligned}$$
	Take the supremum of both sides of the above inequality, we obtain that $$ C^I_{NJ}(\alpha,p,X)\leq\frac12\gamma^p_X(1-2\alpha).$$
On the other hand, let $x_1, x_2 \in S_X$, we choose $v_1 = \frac{x_1+x_2}{2}, v_2 = \frac{x_1-x_2}{2}$, then $v_1 + v_2, v_1 - v_2 \in S_X$. Since
$$\begin{aligned}
	&\frac{\|x_1 - (1-2\alpha)x_2\|^p + \|x_1 + (1-2\alpha)x_2\|^p}{2^{p-1}}\\& = \frac{\|v_1 + v_2 - (1-2\alpha)(v_1 - v_2)\|^p + \|v_1 + v_2 + (1-2\alpha)(v_1 - v_2)\|^p}{2^{p-1}\|v_1+v_2\|^p}\\&= 2 \cdot \frac{\|\alpha v_1 + (1-\alpha)v_2\|^p + \|(1-\alpha)v_1 + \alpha v_2\|^p}{\|v_1+v_2\|^p}\\&\leq 2 C^I_{NJ}(\alpha,p,X),
\end{aligned}$$it follows that $ C^I_{NJ}(\alpha,p,X) \geq \frac{1}{2}\gamma^p_X(1-2\alpha)$, as desired.
\end{proof}
\begin{Example}
Let  $0\leq \alpha\leq \frac12$ and \(2 \leq p < \infty\). Define the linear space  
\[
l^p = \left\{ (x_i)_{i=1}^\infty \subseteq \mathbb{R} \,\middle|\, \sum_{i=1}^\infty |x_i|^p < \infty \right\},
\]  
equipped with the norm  
$\|x\|_p = \left( \sum_{i=1}^\infty |x_i|^p \right)^{1/p}.$ Then $$C^I_{NJ}(\alpha,p,l_p)=(1-\alpha)^p+\alpha^p.$$ 
\end{Example}
\begin{proof}
	On the one hand, since  \(2 \leq p < \infty\) and $0\leq \alpha\leq \frac12$, we have $$\frac{\|x_{1}+tx_{2}\|^{p}+\|x_{1}-tx_{2}\|^{p}}{2^{p-1}}\leq\frac{(1+t)^{p}+(1-t)^{p}}{2^{p-1}},$$ which means that $\gamma^p_{l_p}(t)\leq\frac{(1+t)^{p}+(1-t)^{p}}{2^{p-1}}$.
	
	On the other hand, let $x_1=\bigg(\frac{1}{2^{\frac1p}},\frac{1}{2^{\frac1p}},0,\cdots\bigg),$ and $x_2=\bigg(\frac{1}{2^{\frac1p}},-\frac{1}{2^{\frac1p}},0,\cdots\bigg),$ then we have $$\frac{\|x_{1}+tx_{2}\|^{p}+\|x_{1}-tx_{2}\|^{p}}{2^{p-1}}=\frac{(1+t)^{p}+(1-t)^{p}}{2^{p-1}},$$ this implies that $\gamma^p_{l_p}(t)\geq\frac{(1+t)^{p}+(1-t)^{p}}{2^{p-1}}$. Hence, we obtain that $$\gamma^p_{l_p}(t)=\frac{(1+t)^{p}+(1-t)^{p}}{2^{p-1}}.$$Then use the identity that $ C^I_{NJ}(\alpha,p,X)=\frac12\gamma^p_X(1-2\alpha),$ we get that $$C^I_{NJ}(\alpha,p,l_p)=(1-\alpha)^p+\alpha^p,$$ as desired. 
\end{proof}
\begin{Theorem}\label{t}
	Let $X$ be a Banach space with $1\leq p<\infty$. Then,$$C^p_{NJ}(X)=\sup\bigg\{\frac{2C^I_{NJ}(\frac{1-t}{2},p,X)}{1+t^p},0\leq t\leq 1\bigg\}.$$
\end{Theorem}
\begin{proof}
	From the above Theorem \ref{cr}, we know that $ C^I_{NJ}(\alpha,p,X)=\frac12\gamma^p_X(1-2\alpha).$ Then, let $t=1-2\alpha$, we have $0\leq t\leq 1$ and $\gamma^p_X(t)=2C^p_{I}(\frac{1-t}{2},X)$. Thus, we obtain that$$C^p_{NJ}(X)=\sup\left\{\frac{2C^I_{NJ}(\frac{1-t}{2},p,X)}{1+t^p}:0\leq t\leq1\right\}.$$
\end{proof}
Since the above theorem explains the relationship between \(  C^I_{NJ}(\alpha,p,X) \) and \( C^p_{NJ}(X) \), we can derive 
some examples of the values of the \( C^p_{NJ}(X) \) constant in some specific Banach spaces through this equation.
\begin{Example}
	(i)Let $X_1=(\mathbb R^n,\|\cdot\|_1)$. Then $C^p_{NJ}(X_1)=2.$ 
	
	(ii)Let $X_2=(\mathbb R^n,\|\cdot\|_\infty)$.  Then $C^p_{NJ}(X_2)=2.$ 
\end{Example}
\begin{proof}
In Example \ref{e} and \ref{ee}, we have obtained that 	$C^p_I(\alpha, X_1)=2(1-\alpha)^p$ and $C^I_{NJ}(\alpha,p, X_2)=2(1-\alpha)^p.$ Thus, by Theorem \ref{t}, we obtain that $$C^p_{NJ}(X_i)=\sup\bigg\{\frac{4(1-\frac{1-t}{2})^p}{1+t^p},0\leq t\leq 1\bigg\}(i=1,2).$$ Now we consider the function: for $1\leq p<\infty$ and $0\leq x\leq 1$, define $f(x)=\frac{(1+x)^p}{1+x^p}$, then by a simple calculation, we can get that $f'(x)>0$ is always valid for all $0\leq x\leq 1$, this implies that $1\leq f(x)<2^{p-1}$. Therefore, we obtain that $$C^p_{NJ}(X_i)=2(i=1,2).$$
\end{proof}
\begin{Proposition}
	Let $X$ be a Banach space with $0\leq\alpha\leq\frac12$ and $1\leq p\leq q<\infty$. Then$$C^I_{NJ}(\alpha,q,X)\leq C^I_{NJ}(\alpha,p,X)\leq 2^{1-\frac{p}{q}}(C^I_{NJ}(\alpha,q,X))^{\frac{p}{q}}.$$
\end{Proposition}
\begin{proof}
	To prove the left inequality, we first consider a function: $g(t)=a^t$, where $0<a\leq1$ and $1\leq t<\infty$, it's known that $g(t)$ is a non-decreasing function. Then let $x_1,x_2\in S_X$ and $0\leq t\leq 1$, we have  $$\max\bigg\{\frac{\| x_{1}+t x_{2}\|}{2},\frac{\| x_{1}-t x_{2}\|}{2}\bigg\}\leq\frac{1+t}{2}\leq1.$$
	Thus, if $1\leq p\leq q<\infty$, we obtain that$$\begin{aligned}\frac{\left\|x_{1}+tx_{2}\right\|^{q}+\left\|x_{1}-tx_{2}\right\|^{q}}{2^{q-1}}&=2\bigg(\left\|\frac{x_{1}+tx_{2}}{2}\right\|^{q}+\left\|\frac{x_{1}-tx_{2}}{2}\right\|^{q}\bigg)\\&\leq2\left(\left\|\frac{x_{1}+tx_{2}}{2}\right\|^{p}+\left\|\frac{x_{1}-tx_{2}}{2}\right\|^{p}\right)\\&=\frac{\|x_{1}+tx_{2}\|^{p}+\|x_{1}-tx_{2}\|^{p}}{2^{p-1}}\end{aligned}$$ according to $g(t)$ is a non-decreasing function. This implies that $\gamma^q_X(t)\leq\gamma^p_X(t)$. Hence, by Theorem \ref{cr}, we get that $$C^I_{NJ}(\alpha,p,X)\geq C^I_{NJ}(\alpha,q,X).$$
	Conversely, from $$\begin{aligned}\gamma^p_{X}(t)&=\sup\bigg\{\frac{\|x_{1}+tx_{2}\|^{p}+\|x_{1}-tx_{2}\|^{p}}{2^{p-1}}:x_1,x_2\in S_X,0\leq t\leq1\bigg\}\\&\leq \sup\bigg\{\frac{2^{1-\frac{p}{q}}(\|x_{1}+tx_{2}\|^{q}+\|x_{1}-tx_{2}\|^{q})^{\frac{p}{q}}}{2^{p-1}}:x_1,x_2\in S_X,0\leq t\leq1\bigg\}\\&=2^{1-\frac{p}{q}}\big(\gamma_{X}^{q}(t)\big)^{\frac{p}{q}},\end{aligned}$$and combine with  Theorem \ref{cr}, we obtain that  $$C^I_{NJ}(\alpha,p,X)\leq 2^{1-\frac{p}{q}}(C^I_{NJ}(\alpha,q,X))^{\frac{p}{q}},$$ as desired.
\end{proof}
\begin{Theorem}
	Let $X$ be a Banach space with $0\leq\alpha\leq\frac12$ and $1\leq p<\infty$. Then,$$\frac{1}{2^{p-1}}\bigg(\rho_X(1-2\alpha)+1\bigg)^p\leq C^I_{NJ}(\alpha,p,X)\leq\widetilde{C}^{(p)}_{NJ}(X).$$
\end{Theorem}
\begin{proof}
	On the one hand, since $$\begin{aligned}\gamma_{X}^{p}(t)&\geq\frac{\|x_{1}+tx_{2}\|^{p}+\|x_{1}-tx_{2}\|^{p}}{2^{p-1}}\\&\geq\frac{1}{2^{p-2}}\left(\frac{\|x_{1}+tx_2\|+\|x_{1}-tx_2\|}{2}\right)^{p}\\&\geq\frac{1}{2^{p-2}}(\rho_{X}(t)+1)^{p},\end{aligned}$$ then according to Theorem \ref{cr}, we have $$C^I_{NJ}(\alpha,p,X)\geq\frac{1}{2^{p-1}}\bigg(\rho_X(1-2\alpha)+1\bigg)^p.$$
	On the other hand, since we can obtain that $\gamma_{X}^{p}(t)$ is a non-decreasing function of $t$ on $[0,1]$ in the proof process of Proposition \ref{p1}, which means that\begin{equation}\label{e0}
		\begin{aligned}
			\frac{\|x_1+t x_2\|^p+\|x_1- t x_2\|^p}{2^{p-1}}&\leq\frac{\|x_1+x_2\|^p+\|x_1-x_2\|^p}{2^{p-1}}\\&\leq2\widetilde{C}^{(p)}_{NJ}(X).
		\end{aligned}
	\end{equation}Combine \eqref{e0} and Theorem \ref{cr}, we can get that $$C^I_{NJ}(\alpha,p,X)\leq\widetilde{C}^{(p)}_{NJ}(X).$$
\end{proof}

By means of the following theorem, we establish the relationship between the $C^I_{NJ}(\alpha,p,X)$ constant and the classical constant \( J(X) \).
\begin{Theorem}\label{ttt}
	Let $X$ be a Banach space with $0\leq\alpha\leq\frac12$ and $1\leq p<\infty$. Then,$$ \frac{(J(X)-2\alpha)^p}{2^{p-1}}\leq  C^I_{NJ}(\alpha,p,X)\leq\frac{2^p\alpha^p+2^{2p}(1-2\alpha)^p}{J(X)^p}.$$  
\end{Theorem}
\begin{proof}
	Since$$\begin{aligned}2\min\{\|x_{1}+tx_{2}\|^{p},\|x_{1}-tx_{2}\|^{p}\}&\leq\|x_{1}+tx_{2}\|^{p}+\|x_{1}-tx_{2}\|^{p}\\&\leq 2^{p-1}\gamma_{X}^{p}(t),\end{aligned}$$ which means that $\sqrt[p]{2^{p-2}\gamma_{X}^{p}(t)}\geq \min\left\{\|x_{1}+tx_{2}\|,\|x_{1}-tx_{2}\|\right\}.$ They by $$\|x_1+tx_2\|\geq\|x_1+x_2\|-(1-t),$$ and $$\|x_1-tx_2\|\geq\|x_1-x_2\|-(1-t),$$ we obtain that $$\sqrt[p]{2^{p-2}\gamma_{X}^{p}(t)}\geq\min\{\|x_1+x_2\|,\|x_1-x_2\|\}-(1-t).$$ Thus, we get that $\gamma_{X}^{p}(t)\geq\frac{[J(X)-(1-t)]^{p}}{2^{p-2}}$. Then by $ C^I_{NJ}(\alpha,p,X)=\frac12\gamma^p_X(1-2\alpha)$ in Theorem \ref{cr}, it follows that $$ C^I_{NJ}(\alpha,p,X)\geq\frac{(J(X)-2\alpha)^p}{2^{p-1}},$$ as desired.

	Conversely, for any $x,y\in S_X$ such that $x \perp_I x_2$, we have $$\begin{aligned}&\frac{\|\alpha x_1+(1-\alpha)x_2\|^{p}+\|(1-\alpha)x_1+\alpha x_2\|^{p}}{\|x_1+x_2\|^{p}}\\=&\frac{\|\alpha(x_1+x_2)+(1-2\alpha)x_2\|^{p}+\|\alpha(x_1+x_2)+(1-2\alpha)x_1\|^{p}}{\|x_1+x_2\|^{p}}\\\leq&\frac{2^{p}\big[\alpha^{p}\|x_1+x_2\|^{p}+(1-2\alpha)^{p}\big]}{\|x_1+x_2\|^{p}}\\\leq&2^p\alpha^p+2^p(1-2\alpha)^pS(X)^p.\end{aligned}$$ Then according to $J(X)S(X)=2$, we obtain that$$ C^I_{NJ}(\alpha,p,X)\leq\frac{2^p\alpha^p+2^{2p}(1-2\alpha)^p}{J(X)^p}.$$
\end{proof}
	\section{Relations with $C_{N J}^I(\alpha, p, X)$ and other geometric properties  }
	In this section, we will discuss some relations with $C_{N J}^I(\alpha, p, X)$ and other geometric properties, first, we will demonstrate the relationship between the $C_{N J}^I(\alpha, p, X)$  constant and uniform non-squareness. Before we begin, we need the following lemma.
\begin{Lemma}\label{ll}\cite{02}
Let $X$ be a Banach space, $x,y \in X$, and $x \perp_I y$. Then:

(i) For $|\alpha| \geq 1$,
$$\|x \pm y\| \leq \|x + \alpha y\| \leq |\alpha| \|x \pm y\|.$$

(ii) For $|\alpha| \leq 1$,
$$
|\alpha| \|x \pm y\| \leq \|x + \alpha y\| \leq \|x \pm y\|.$$
\end{Lemma}
\begin{Theorem}\label{t4}
	Let $X$ be a finite dimensional Banach space, then $X$ is not uniformly non-square
	if and only if $ C^I_{NJ}(\alpha,p,X)=2(1-\alpha)^p.$
\end{Theorem}
\begin{proof}
	First, if $X$ is not uniformly non- square,  then there exists $x_{1n}, x_{2n}\in S_{X}$ such that	
	\begin{equation}\label{e2}
		\|x_{1n}+x_{2n}\|\to2,\|x_{1n}-x_{2n}\|\to2(n\to\infty).
	\end{equation}
	Let $u_1 = \frac {x_{1n}+ x_{2n}}2, u_2= \frac {x_{1n}- x_{2n}}2$, then we have $\|u_1+ u_2 \| = \| u_1 - u_2 \| = 1.$ Since
$$\begin{aligned}\|\alpha u_{1}+(1-\alpha)u_{2}\|&=\bigg\|\alpha\cdot\frac{x_{1n}+x_{2n}}{2}+(1-\alpha)\cdot\frac{x_{1n}-x_{2n}}{2}\bigg\|\\&=\bigg\|\frac{x_{1n}}{2}+\frac{2\alpha-1}{2}x_{2n}\bigg\|\\&=\bigg\|\frac{x_{1n}-x_{2n}}{2}+\alpha x_{2n}\bigg\|\\&\geq\frac{\|x_{1n}-x_{2n}\|}{2}-\alpha\|y_{n}\|\end{aligned}$$
	and
$$\begin{aligned}\|(1-\alpha) u_{1}+\alpha u_{2}\|&=\bigg\|(1-\alpha)\cdot\frac{x_{1n}+x_{2n}}{2}+\alpha\cdot\frac{x_{1n}-x_{2n}}{2}\bigg\|\\&=\bigg\|\frac{x_{1n}}{2}+\frac{2\alpha-1}{2}x_{2n}\bigg\|\\&\geq\frac{\|x_{1n}-x_{2n}\|}{2}-\alpha\|y_{n}\|.\end{aligned}$$Thus, we obtain that $\frac{\|x_{1n}-x_{2n}\|}{2}-\alpha\|y_{n}\|\to1-\alpha(n\to\infty).$ Then by the fact that $$\begin{aligned}\|\alpha u_{1}+(1-\alpha)u_{2}\|&=\|\alpha(u_{1}+u_{2})+(1-2\alpha)u_{2}\|\\&\leq\alpha\|u_{1}+u_{2}\|+1-2\alpha\\&=1-\alpha\end{aligned}$$ and $$\begin{aligned}\|(1-\alpha)u_{1}+\alpha u_{2}\|&=\|\alpha(u_{1}+u_{2})+(1-2\alpha)u_{1}\|\\&\leq\alpha\|u_{1}+u_{2}\|+1-2\alpha\\&=1-\alpha.\end{aligned}$$ We can get that $$\|\alpha u_{1}+(1-\alpha)u_{2}\|\to1-\alpha(n\to\infty)$$ and  $$\|(1-\alpha) u_{1}+\alpha u_{2}\|\to1-\alpha(n\to\infty),$$ it follows that $ C^I_{NJ}(\alpha,p,X)=2(1-\alpha)^p,$ as desired.

On the other hand, from $ C^I_{NJ}(\alpha,p,X)=2(1-\alpha)^p,$ we can deduce that  there exist $x_{1n} \in S_X$, $x_{2n} \in B_X$ such that $x_{1n} \perp_I x_{2n}$ and
$$\lim_{n \to \infty} \frac{\|\alpha x_{1n} + (1 - \alpha)x_{2n}\|^p + \|(1 - \alpha)x_{1n} + \alpha x_{2n}\|^p}{\|x_{1n} + x_{2n}\|^p} = 2(1-\alpha)^p.$$
From $X$ is a finite dimensional Banach space, we can deduce that there exist $\widetilde{x_1}$, $\widetilde{x_2} \in B_X$ such that $\widetilde{x_1} \perp_I \widetilde{x_2}$ and
$$
\lim_{p \to \infty} \|x_{1n_p}\| = \|\widetilde{x_1}\|, \lim_{p \to \infty} \|x_{2n_p}\| = \|\widetilde{x_2}\|.$$ By Lemma \ref{ll}, we have $$\|\alpha x_{1n}+(1-\alpha)x_{2n}\|\leq(1-\alpha)\|x_{1n}+x_{2n}\|$$and $$\|(1-\alpha)x_{1n}+x_{2n}\|\leq(1-\alpha)\|x_{1n}+x_{2n}\|$$ are always valid for all $0\leq\alpha\leq\frac12$.  Since $$\frac{\bigg((1-\alpha)\|x_{1n}+x_{2n}\|\bigg)^p+\bigg((1-\alpha)\|x_{1n}+x_{2n}\|\bigg)^p}{\|x_{1n}+x_{2n}\|^p}=2(1-\alpha)^p,$$ we can get that \begin{equation}\label{e4}
\|\alpha \widetilde{x_1}+(1-\alpha)\widetilde{x_2}\|=(1-\alpha)\|\widetilde{x_1}+\widetilde{x_2}\|
\end{equation}and \begin{equation}\label{e5}
\|(1-\alpha)\widetilde{x_1}+\widetilde{x_2}\|=(1-\alpha)\|\widetilde{x_1}+\widetilde{x_2}\|.
\end{equation}Then by \eqref{e4} and the fact that $\|\alpha \widetilde{x_1}+(1-\alpha)\widetilde{x_2}\|\leq(1-2\alpha)\|\widetilde{x_2}\|+\alpha\|\widetilde{x_1}+\widetilde{x_2}\|$, we have $\|\widetilde{x_1}+\widetilde{x_2}\|\leqslant\|\widetilde{x_2}\|$. Similarly, we can get that  $\|\widetilde{x_1}+\widetilde{x_2}\|\leqslant\|\widetilde{x_1}\|$. Therefore,
$$\max\{\|\widetilde{x_1}+\widetilde{x_2}\|,\|\widetilde{x_1}-\widetilde{x_2}\|\}=\|\widetilde{x_1}+\widetilde{x_2}\|\leqslant\min\{\|\widetilde{x_1}\|,\|\widetilde{x_2}\|\}\leqslant1<1+\delta$$
for any $\delta\in(0,1)$, which shows that $X$ is not uniformly non-square.
\end{proof}
\begin{Lemma}\cite{01}\label{i}
	Let $X$ be a Banach space. Then $X$ is uniformly smooth if and only if $$\lim\limits_{t\to0}\frac{J_{X,p}(t)-1}{t}=0.$$
\end{Lemma}
\begin{Theorem}
	Let $X$ be a Banach space. Then $X$ is uniformly smooth if and only if $$\lim_{\alpha\to\frac{1}{2}}\frac{\sqrt[P]{2^{p-1}C^I_{NJ}(\alpha,p,X)}-1}{1-2\alpha}=0.$$
\end{Theorem}
\begin{proof}
 Let $t= 1-2\alpha$, then $\alpha\to\frac12$ is equivalent to $t\to0.$ By Theorem \ref{cr}, we
can get that
$$\begin{aligned}\frac{\sqrt[P]{2^{p-1}C^I_{NJ}(\alpha,p,X)}-1}{1-2\alpha}&=\frac{\sqrt[P]{2^{p-2}\gamma^p_{X}(1-2\alpha)}-1}{1-2\alpha}\\&=\frac{J_{X,p}(t)-1}{t}.\end{aligned}$$
Therefore, $\lim\limits_{t\to0}\frac{J_{X,p}(t)-1}{t}=0$ if and only if $\lim\limits_{\alpha\to\frac{1}{2}}\frac{\sqrt[P]{2^{p-1}C^I_{NJ}(\alpha,p,X)}-1}{1-2\alpha}=0.$ That is
$$\lim\limits_{t\to0}\frac{J_{X,p}(t)-1}{t}=0$$
if and only if X is uniformly smooth by Lemma \ref{i}.
\end{proof}


\begin{thebibliography}{00}
		\bibitem{01}Zuo, Z.,  Cui, Y. A coefficient related to some geometric properties of a Banach space. J. Inequa. Appl., 2009, 1-14.
		\bibitem{02}James, R.C. (1945). Orthogonality in normed linear spaces. Duke Math. J. 12(2): 291–302.
		\bibitem{03}Amini-Harandi, Alireza, Ian Doust, and Gavin Robertson. Roundness properties of Banach spaces. J. Funct. Anal. 281.10 (2021): 109230.
	    \bibitem{04}J. Gao and K. S. Lau, On two classes of Banach spaces with uniform normal structure,
		Studia Math. 99 (1991), 41–56.
		\bibitem{05}M. Kato, L. Maligranda and Y. Takahashi, On James and Jordan–von Neumann constants and the normal structure coefficient of Banach spaces, Studia Math. 144 (2001),
		275–295.
		\bibitem{06}Clarkson, J.A. (1937). The von Neumann-Jordan constant for the Lebesgue space. Ann. of Math. 38: 114–115.
		\bibitem{07}J. A. Clarkson, Uniformly convex spaces, Trans. Am. Math. Soc., 40(3)(1936), 396–
		414.
		\bibitem{08}Ji GAO, LAU. K. S. On the geometry of spheres in normed linear spaces. J. Aust.Math. Soc. Ser. A., 1990,
		48(1): 101-112.
		\bibitem{09}Birkhoff, G.: Orthogonality in linear metric spaces. Duke Math. J. 1(2), 169–172 (1935)
		\bibitem{10}J. Lindenstrauss, On the modulus of smoothness and divergent series in Banach spaces, Michigan Math. J. 10 (1963)
		241–252
		\bibitem{12}Roberts, B.D.: On the geometry of abstract vector spaces. Tohoku Math. J. 39, 42–59 (1934)
		\bibitem{13}Ji D, Wu S. Quantitative characterization of the difference between Birkhoff orthogonality and isosceles orthogonality. J. Math.
		Anal. Appl., 2006, 323(1): 1-7.
		\bibitem{14}Baronti M, Casini E, Papini P L. Isosceles constant in Banach spaces. J. Math. Anal. Appl., 2024, 536(1): 128251.
		\bibitem{15}Baronti M, Bertella V. A Generalization of the Isosceles Constant in Banach Spaces. Mediterr. J. Math., 2024, 21(3): 113.
		\bibitem{16}Y. Cui, W. Huang, H. Hudzik and R. Kaczmarek, Generalized von Neumann–Jordan constant and its relationship to the fixed point property, Fixed Point Theory Appl. 2015,
		no. 1, 40.
		\bibitem{17}R. C. James, Uniformly non-square Banach spaces. Ann. of Math. 80 (1964), 542–550.
		\bibitem{18}Qi Liu, Zhijian Yang, Yongjin Li. New geometric constants of isosceles orthogonal type.
		e-print arXiv: 2111.08392. 
		\bibitem{19}Qichuan Ni, Qi Liu, Yuxin Wang, Jinyu Xia, Ranran Wang. Symmetric form geometric constant related to isosceles orthogonality in Banach spaces., submitted.
		\bibitem{20}Yang, C. S., Wang, F.H. (2006). On a new geometric constant related to the von Neumann-Jordan constant. J. Math.
		Anal. Appl. 324(1): 555-565.
		\bibitem{21}Zuo Z, Cui Y. A coefficient related to some geometric properties of a Banach space. J. Inequal. Appl., 2009: 1-14.
		\bibitem{22}Alonso J, Martini H, Wu S. On Birkhoff orthogonality and isosceles orthogonality in normed linear spaces. Aequationes math., 2012, 83(1): 153-189.
	\end{thebibliography}
\end{document}